\documentclass[11pt]{amsart}
\usepackage{amsfonts}
\usepackage{amsmath}
\usepackage{amsthm}
\usepackage{url}
\usepackage[all]{xy}
\usepackage{latexsym}
\usepackage{amssymb}
\usepackage[cp850]{inputenc}
\usepackage{amsfonts}
\usepackage{graphicx}
\usepackage{dsfont}
\usepackage{float}

\usepackage[usenames]{color}

\newtheorem{sat}{Theorem}[section]		
\newtheorem{lem}[sat]{Lemma}

\newtheorem*{defi*}{Definition}			
\newtheorem*{bei*}{Example}
\newtheorem*{sat*}{Theorem}				
\newtheorem*{kor*}{Corollary}
\newtheorem*{rmk*}{Remark}				
	
\newtheorem*{quest*}{Question}	
	
\newtheorem*{claim}{Claim}	

\let\ssection=\section
\renewcommand{\section}{\setcounter{equation}{0}\ssection}

\newtheorem*{namedtheorem}{\theoremname}
\newcommand{\theoremname}{testing}
\newenvironment{named}[1]{\renewcommand{\theoremname}{#1}\begin{namedtheorem}}{\end{namedtheorem}}

\theoremstyle{remark}
\newtheorem*{bem}{Remark}

\newtheorem*{namedtheoremr}{\theoremnamer}
\newcommand{\theoremnamer}{testing}

\newcommand{\BR}{\mathbb R}			
			
			\newcommand{\BZ}{\mathbb Z}

\newcommand{\CA}{\mathcal A}		
\newcommand{\CC}{\mathcal C}		
\newcommand{\CE}{\mathcal E}		
\newcommand{\CG}{\mathcal G}		
		
\newcommand{\CK}{\mathcal K}		\newcommand{\CL}{\mathcal L}
\newcommand{\CM}{\mathcal M}		
		\newcommand{\CP}{\mathcal P}
		
\newcommand{\CS}{\mathcal S}		
\newcommand{\CU}{\mathcal U}

\newcommand{\FM}{\mathfrak m}

\newcommand{\FN}{\mathfrak n}
\newcommand{\FB}{\mathfrak b}

\newcommand{\D}{\partial}

\DeclareMathOperator{\Id}{Id}		%	Identit\"at
\newcommand{\comment}[1]{}

\DeclareMathOperator{\Lip}{Lip}
\DeclareMathOperator{\Thu}{Thu}

\newcommand{\fsubd}{\mathrel{{\scriptstyle\searrow}\kern-1ex^d\kern0.5ex}}
\newcommand{\bsubd}{\mathrel{{\scriptstyle\swarrow}\kern-1.6ex^d\kern0.8ex}}

\renewcommand{\epsilon}{\varepsilon}
\renewcommand{\le}{\leqslant}
\renewcommand{\ge}{\geqslant}

\begin{document}

\title[]{Distribution in the unit tangent bundle of the geodesics of given type}
  \author{Viveka Erlandsson}
  \address{School of Mathematics, University of Bristol \\ Bristol BS8 1UG, UK {\rm and}  \newline ${ }$ \hspace{0.2cm} Department of Mathematics and Statistics, UiT The Arctic University of  \newline ${ }$ \hspace{0.2cm} Norway}
  \email{v.erlandsson@bristol.ac.uk}
\thanks{The first author gratefully acknowledges support from EPSRC grant EP/T015926/1.}
\author{Juan Souto}
\address{UNIV RENNES, CNRS, IRMAR - UMR 6625, F-35000 RENNES, FRANCE}
\email{jsoutoc@gmail.com}

\begin{abstract}
Recall that two geodesics in a negatively curved surface $S$ are of the same type if their free homotopy classes differ by a homeomorphism of the surface. In this note we study the distribution in the unit tangent bundle of the geodesics of fixed type, proving that they are asymptotically equidistributed with respect to a certain measure $\FM^S$ on $T^1S$. We study a few properties of this measure, showing for example that it distinguishes between hyperbolic surfaces.
\end{abstract}
\dedicatory{Dedicated to Scott Wolpert.}

\maketitle

\section{}

Let $S$ be a closed, orientable, connected surface of genus $g$ endowed with a negatively curved metric and let $\rho^S=(\rho_t^S)_{t\in\BR}$ be the associated geodesic flow on the unit tangent bundle $T^1S$. We can associate to every primitive periodic $\rho^S$-orbit $\gamma$, that is to every non-trivial closed primitive oriented geodesic, a measure $\vec\gamma$ on $T^1S$ as follows: Choose $v\in T^1S$ tangent to $\gamma$, let $\ell_S(\gamma)$ be the length of $\gamma$, that is the period of the orbit $\gamma$, and set
\begin{equation}\label{eq:bla}
  \int_{T^1S} f\, d\vec\gamma\stackrel{\text{def}}=\int_0^{\ell_S(\gamma)}f(\rho_t(v))\, dt
\end{equation}
for all $f\in C^0(T^1S)$. In \cite{Bowen72} Bowen investigated (for general hyperbolic flows) the distribution of the set $\CP_L(S)$ of all primitive periodic $\rho^S$-orbits with period at most $L$. To do so he studied the behavior, when $L$ tends to $\infty$, of the measures
\begin{equation}\label{eq:bowen}
  b^S_L=\sum_{\gamma\in\CP_L(S)}\vec\gamma
\end{equation}
and proved that the associated probability measures converge to the measure of maximal entropy of $\rho^S$:
$$\FB^S=\lim_{L\to\infty}\frac 1{\Vert b^S_L\Vert}b^S_L.$$

In this note we study what happens if we condition Bowen's construction to the set of periodic orbits of a give type. Here we say, \`a la Mirzakhani, that two closed geodesics are {\em of the same type} if their unoriented free homotopy classes differ by a homeomorphism of the surface---for example, any two non-separating simple curves are of the same type. Now, if $\gamma_0$ is a closed primitive geodesic in $S$ we consider the asymptotic behaviour for $L\to\infty$ of the measures
\begin{equation}\label{eq:bowen2}
  m_L^{S,\gamma_0}=\sum_{\gamma\in\CP_L(S,\gamma_0)}\vec\gamma
\end{equation}
where $\CP_L(S,\gamma_0)$ is the set of all primitive periodic $\rho^S$-orbits of type $\gamma_0$ and with at most length $L$. We prove:

\begin{sat}\label{main}
  Let $S$ be a closed orientable surface of genus $g$ endowed with a negatively curved metric. There is a measure $\FM^S$ on $T^1S$, invariant under both the geodesic flow and the geodesic flip, and with
  $$\lim_{L\to\infty}\frac 1{\Vert m_L^{S,\gamma_0}\Vert}m_L^{S,\gamma_0}=\FM^S$$
  for every non-trivial closed primitive geodesic $\gamma_0$ in $S$.
\end{sat}

The measures $b_L^S$ and $m_L^{S,\gamma_0}$ are very different. For instance, the total measure of the former grows exponentially when $L\to\infty$ while, as we will see in Lemma \ref{lem:total measure}, the latter has total measure asymptotic to a multiple of $L^{6g-5}$. Still, since by Bowen's theorem the set of all geodesics accumulates to the measure of maximal entropy $\FB^S$, and since the measure $\FM^S$ in Theorem \ref{main} is independent of the type of geodesic under consideration, it could maybe sound reasonable to conjecture that $\FM^S$ is once again the measure of maximal entropy. This is definitively not the case. For example, the measure of maximal entropy of the geodesic flow is ergodic, has positive entropy, and has full support in $T^1S$. None of this is true for the measure $\FM^S$:

\begin{sat}\label{properties of measure}
The measure $\FM^S$ is not ergodic, has vanishing entropy, and its support has Hausdorff dimension $1$. 
\end{sat}

Every homeomorphism $\phi:S\to S'$ between two closed negatively curved surfaces induces an orbit equivalence between the associated geodesic flows. We get then a homeomorphism $\phi_*$---see \eqref{eq-homeo2} below---between the spaces of geodesic flow invariant measures on $T^1S$ and $T^1S'$ respectively; moreover $\phi_*$ only depends on the homotopy class of $\phi$. Anyways, $\phi_*$ allows us to compare the measures $\FB^S$ and $\FB^{S'}$, and it has been conjectured that they are mutually singular unless $\phi_*$ is isotopic to a homothety, an isometry if both surfaces have the same area. In fact, in \cite{Otal90} Otal proved that this conjecture holds if one replaces the measure of maximal entropy by the Liouville measure. We prove that again from this point of view the measure $\FM^S$ behaves in a very different way than the measure of maximal entropy:

\begin{sat}\label{theorem 3}
If $\phi:S\to S'$ is a homeomorphism between closed negatively curved surfaces then the measures $\phi_*(\FM^S)$ and $\FM^{S'}$ are in the same measure class.
\end{sat}

It is maybe worth pointing out that in the course of the proof of Theorem \ref{theorem 3} we will give an explicit formula for the Radon-Nikodym derivative of $\phi_*(\FM^{S})$ with respect to $\FM^{S'}$.

While, by Theorem \ref{theorem 3}, the measure class of $\FM^S$ is independent of the particular metric, we prove, as mentioned in the abstract, that the measure $\FM^S$ is itself rich enough to distinguish between hyperbolic metrics:

\begin{sat}\label{theorem 4}
  A homeomorphism $\phi:S\to S'$ between closed orientable hyperbolic surfaces is isotopic to an isometry if and only if $\phi_*(\FM^S)=\FM^{S'}$.
\end{sat}

One should keep in mind that Theorem \ref{theorem 4} fails for variable curvature metrics. Indeed, as we will see in Section \ref{sec8}, there are $S$ and $S'$ negatively curved, with the same area, and non-isometric to each other, but such that there is a homeomorphism $\phi:S\to S'$ with $\phi_*\big(\FM^S)=\FM^{S'}$. Moreover here the surface $S$ can be chosen to be hyperbolic.
\medskip

After having presented the results of this paper, let us give a brief indication of the tools used in the proofs. We rely heavily on the work of Mirzakhani \cite{Maryam1,Maryam2} on counting geodesics of a given type $\gamma_0$ and with at most length $L$. More concretely she proved that
$$\vert\{\gamma\text{ of type }\gamma_0\text{ and with }\ell_S(\gamma)\le L\}\vert\approx C\cdot L^{6g-6}$$
where $C=C(\gamma_0,S)$ is a constant and where $\approx$ indicates that the ratio between both sides tends to 1 as $L$ grows. In  \cite{book} we recovered this result proving that the limit
\begin{equation}\label{eq-measureCconverge}
  \lim_{L\to\infty}\frac 1{L^{6g-6}}\sum_{\gamma\sim\gamma_0}\delta_{\frac 1L\gamma}
\end{equation}
exists on the space of measures on the space $\CC(S)$ of currents on $S$---here $\delta_x$ is the Dirac measure centred at $x$ and $\gamma\sim\gamma_0$ means that they both have the same type. The relation between the existence of this limit and the results mentioned earlier is via the fact that the space of currents on $S$ is naturally homeomorphic to the space of measures on $T^1S$ which are invariant under the geodesic flow and the geodesic flip $v\mapsto -v$. 
\medskip

In Section \ref{sec2} we recall what currents are, as well as the homeomorphism between the spaces of currents and of flip and flow invariant measures. In Section \ref{sec3} we recall what is the Thurston measure on the space of measured laminations, discussing briefly how to write it in polar coordinates and stating \eqref{eq-measureCconverge} precisely. In Section \ref{sec5}, the bulk of the paper, we prove Theorem \ref{main}. In Section \ref{sec6}, Section \ref{sec7}, and Section \ref{sec8} we prove the other three theorems mentioned above. Finally, in Section \ref{sec9} we discuss briefly what happens if the surface has punctures.

\subsection*{Comments}

Before moving on let us comment briefly on the assumptions on the surface $S$ and on a subtility in what we do here.

\subsubsection*{Punctures.}
The reader might be wondering how important is the fact that we are dealing with a closed surface. Not at all. All theorems hold as stated and we only consider the case of closed surfaces to improve the readability of the paper. Indeed, while it is not hard to work with currents on surfaces which are not closed, it definitively needs more attention to details and one needs to keep making sure to point out that everything happens on compact sets. As announced above, we will discuss in Section \ref{sec9} the changes needed to extend the results to the case with punctures.

\subsubsection*{Orientability}
As we used to do ourselves, the reader might well embrace the belief that non-orientable surfaces are an urban myth, or rather that everything which is true for orientable surfaces is, after minimal thinking, also true for non-orientable surfaces. Well, in this area things are definitively different. For example, it is due to Gendulphe \cite{Guendulphe} that there is no integer with which one can replace $6g-6$ so that \eqref{eq-measureCconverge} still holds if the surface is not orientable. This means that our proof really uses the fact that $S$ is orientable, and that we do not know what to do in the non-orientable setting.
  
\subsubsection*{Time reversal}
Since the geodesic flow commutes with the geodesic flip we get that the measures $b^S_L$ in \eqref{eq:bowen} are not only flow invariant but also flip invariant. We then defined the measures $m_L^{S,\gamma_0}$ so that they are also flip invariant, referring to the {\em type} of a geodesic instead of its {\em oriented type}. So far, all seems identical. Well, there is a difference. Bowen's result applies to Axiom A flows, and in particular to any perturbation of the geodesic flow, flip invariant or not. In this case the flip invariance of the measures \eqref{eq:bowen} is lost, but Bowen's theorem holds. Flip invariance is however key for us. The reason is that currents, as least in the setting needed to establish \eqref{eq-measureCconverge}, are inherently flip invariant. This means we do not know what to do if instead of working with {\em type} we work with {\em oriented type}.

\subsection*{Acknowledgements}
We are grateful to Scott Wolpert and especially to Steven Zelditch for not only asking the question leading to this paper, but also for the ensuing discussions. All of this happened after a talk by the first author at the conference {\em Analysis and Geometry--A symposium to honor the 70th birthday of Scott Wolpert}, and she thanks the organizers for the invitation.

\section{}\label{sec2}
In this section we recall a few definitions and facts about currents.  We refer to \cite{AL,BonahonFrench,Bonahon2,Bonahon,book} for background and details.
\medskip

Denoting by $S$ a closed surface of negative curvature, let $\tilde S$ be its universal cover and $\CG(\tilde S)$ the set of unoriented geodesics therein. Since geodesics in $\tilde S$ are determined by their end points, we have an identification
$$\CG(\tilde S)=(\D_\infty\tilde S\times\D_\infty\tilde S\setminus\Delta)/_{\text{flip}}$$
where $\D_\infty\tilde S$ is the boundary at infinity and where $\Delta$ is the diagonal. The fundamental group $\pi_1(S)$ acts by deck transformation on $\tilde S$ and hence on $\CG(\tilde S)$. A {\em (geodesic) current} on $S$ is a $\pi_1(S)$-invariant Radon measure on $\CG(\tilde S)$. There are plenty of currents. For example, if $\gamma$ is a closed unoriented geodesic in $S$ then 
$$\widehat\gamma(U)=\text{number of lifts of }\gamma\text{ belonging to }U\text{ for }U\subset\CG(\tilde S)$$
defines a current, the so-called {\em counting current associated to }$\gamma$. The set $\CC(S)$ of all currents in $S$ is, when endowed with the weak-*-topology, a cone in a linear space and multiples of counting currents are dense therein \cite{BonahonFrench,book}.

It will be important for us that homeomorphisms $\phi:S\to S'$ between closed negatively curved surfaces induce homeomorphisms between the associated spaces of currents. Indeed, it follows for instance from the Milnor-Svarc lemma that their lifts $\tilde\phi:\tilde S\to\tilde S'$ to the universal cover are quasi-isometries. This implies in particular that they have extensions
$$\D_\infty\tilde\phi:\D_\infty\tilde S\to\D_\infty\tilde S'$$
to the boundary at infinity---extensions which are equivariant under the same homomorphism $\phi_*:\pi_1(S)\to\pi_1(S')$ with respect to which the chosen lift $\tilde\phi$ was equivariant.

The boundary map $\D_\infty\tilde\phi$ induces an again equivariant map between the sets of unoriented geodesics
$$\phi_*:\CG(\tilde S)\to\CG(\tilde S')$$
and hence a map, actually a homeomorphism,
\begin{equation}\label{eq-currents homeo}
  \phi_*:\CC(S)\to\CC(S')
\end{equation}
between the corresponding spaces of currents. Note that this homeomorphism just depends on the free homotopy class of the map $\phi$ and that it satisfies
  $$\phi_*(\widehat\gamma)=\widehat{\phi(\gamma)}$$
for every curve $\gamma$. In words, $\phi_*$ maps the counting current associated to the geodesic $\gamma$ to the counting current associated to the geodesic freely homotopic to $\phi(\gamma)$.
\medskip

The space $\CC(S)$ of currents on $S$ is closely related to the space $\CM_{\text{flip-flow}}(S)$ of Radon measures on $T^1S$ which are invariant both under the geodesic flow and the geodesic flip. Indeed, every current $\widehat\mu\in\CC(S)$ induces a flip and flow invariant measure $\overline\mu=dt\otimes\widehat\mu$ on $T^1S$ as follows: If $Z\subset T^1\tilde S$ is a transversal to the geodesic flow on the universal cover meeting each geodesic at most once then the $\overline\mu$-measure of the flow box
$$[0,t_0]\times Z=\cup_{t\in[0,t_0]}\rho_t^{\tilde S}(Z)\subset T^1\tilde S$$
is given by
\begin{equation}\label{eq-product}
  \overline\mu([0,t_0]\times Z)=\frac 12 t_0\cdot\widehat\mu(Z)
\end{equation}
where $\widehat\mu(Z)$ is the $\widehat\mu$-measure of the subset of $\CG(\tilde S)$ represented by the orbits $t\mapsto\rho_t^{\tilde S}(v)$ with $v\in Z$. The measure $\overline\mu=dt\otimes\widehat\mu$ is $\pi_1(S)$-invariant and hence descends to a measure which we still denote by $\overline\mu=dt\otimes\widehat\mu$ on $T^1S$. This is, by construction, flip and flow invariant. In other words, we have given a map
\begin{equation}\label{eq-homeo}
  \CC(S)\to\CM_{\text{flip-flow}}(S),\ \ \widehat\mu\mapsto\overline\mu=dt\otimes\widehat\mu.
\end{equation}
It is well-known that this map is a homeomorphism \cite{AL,book}. 

\begin{quote}
  {\bf Notation.} As we have been doing so far, we will decorate currents with `hats' and flip and flow invariant measures by `bars'. If we write $\widehat\lambda$ and $\overline\lambda$ then the former is a current, the latter is a flip and flow invariant measure, and one is the image of the other under the homeomorphisms \eqref{eq-homeo}.
\end{quote}

We now explain the pesky $\frac 12$ in \eqref{eq-product}. The flip and flow invariant measure $\bar\gamma$ associated to the counting current $\widehat\gamma$ is given by
$$\overline\gamma=\frac 12(\vec\gamma+\vec\gamma_{\text{\tiny flip}})$$
where we have chosen an orientation of $\gamma$, where $\vec\gamma$ is as in \eqref{eq:bla}, and where $\gamma_{\text{\tiny flip}}$ stands for the oppositly oriented orbit. The factor $\frac 12$ is thus there to guarantee that the measure $\overline\gamma$ has total measure $\ell_S(\gamma)$ for every curve $\gamma$:
$$\Vert\overline\gamma\Vert=\Vert\vec\gamma\Vert=\ell_S(\gamma).$$
Note also that these $\frac 12$ factors disappear when we normalize measures to be probability measures. For example, if we set
\begin{equation}\label{eq-set of geodesics}
  \CS_L(S,\gamma_0)=\{\text{unoriented geodesics }\gamma\text{ of type }\gamma_0\text{ with }\ell_S(\gamma)\le L\}
\end{equation}
then the map $\CP_L(S,\gamma_0)\to\CS_L(S,\gamma_0)$ is two to one: both $\gamma$ and $\gamma_{\text{\tiny flip}}$, and only those, get mapped to the same unoriented geodesic. This means that the measure
\begin{equation}\label{eq:bowen2new}
  \overline m^{S,\gamma_0}_L=\sum_{\gamma\in\CS_L(S,\gamma_0)}\overline\gamma
\end{equation}
is exactly one half of the measure $m^{S,\gamma_0}_L$ defined in \eqref{eq:bowen2}:
$$\overline m^{S,\gamma_0}_L=\frac 12\cdot m^{S,\gamma_0}_L.$$
We get thus the following fact that we record here for later use:

\begin{lem}
For any $\gamma_0$ and any $L$ large enough we have
\begin{equation}\label{eq ricky gervais}
  \frac 1{\Vert\overline m^{S,\gamma_0}_L\Vert}\overline m^{S,\gamma_0}_L=\frac 1{\Vert m^{S,\gamma_0}_L\Vert} m^{S,\gamma_0}_L
\end{equation}
where $m^{S,\gamma_0}_L$ is as in \eqref{eq:bowen2} and $\overline m^{S,\gamma_0}_L$ is as in \eqref{eq:bowen2new}.\qed
\end{lem}

Anyways, the reader might well be thinking that those $\frac 12$ factors are maybe not that painful, and might indeed agree that they are necessary because one is working with unoriented geodesics... But why is one working with unoriented geodesics in the first place? The main reason is that, other than the counting currents $\widehat\gamma$ associated to closed geodesics, the main currents for us are measured laminations, and measured laminations are by their own nature unoriented.

\section{}\label{sec3}
Let $\CM\CL(S)$ be the set of measured laminations on $S$. We can consider measured laminations as currents. Indeed $\CM\CL(S)$ is nothing other than the set of currents whose support projects to a lamination of $S$, that is a compact set foliated by disjoint simple geodesics. As for curves, we will denote by $\widehat\lambda\in\CC(S)$ the current associated to the measured lamination $\lambda\in\CM\CL(S)$. We refer to \cite{AL,Bonahon2,book} for basic facts about measured laminations and their relation to currents.

What will be very important for us is that the set $\CM\CL(S)$ of all measured laminations on $S$ is naturally endowed with a measure, the {\em Thurston measure} $\FM_{\Thu}$. There are different ways of obtaining the Thurston measure. For example one can get it via the standard symplectic structure on the space of measured laminations, but the more natural way here is as a scaling limit
\begin{equation}\label{eq:thurston measure}
  \FM_{\Thu}=\lim_{L\to\infty}\frac 1{L^{6g-6}}\sum_{\gamma\in\CM\CL_\BZ(S)}\delta_{\frac 1L\widehat\gamma}
\end{equation}
where $g$ is the genus of $S$, where $\CM\CL_\BZ(S)$ is the set of integrally weighted simple multicurves on $S$, where we denote by $\widehat\gamma$ the counting current associated to the multicurve $\gamma$, and where the limit is taken with respect to the weak-*-topology on the space of currents. See for example \cite{Maryam1,book} for the construction of the Thurston measures and \cite{MT} for the relations between the two possible constructions we just mentioned. Anyways, the reason why we care about the Thurston measure is that we have the following theorem:

\begin{sat}\label{measure theorem}
  For every closed geodesic $\gamma_0$ in $S$ there is a positive constant $C(\gamma_0)$ such that
  $$\lim_{L\to\infty}\frac 1{L^{6g-6}}\sum_{\gamma\sim\gamma_0}\delta_{\frac 1L\widehat\gamma}=C(\gamma_0)\cdot\FM_{\Thu}.$$
Here $\delta_x$ stands for the Dirac measure centred at $x$, the sum is taken over all $\gamma$ of type $\gamma_0$, and the convergence takes place with respect to the weak-*-topology on the space $\CC(S)$.
\end{sat}

\begin{bem}
  Theorem \ref{measure theorem} is due to Mirzakhani \cite{Maryam1} when $\gamma_0$ is a simple multi\-curve. For general curves, it follows for example when we combine \cite{ES} with \cite{Maryam2}, another result of Mirzakhani. A complete independent proof of Theorem \ref{measure theorem} appears in \cite{book}.
  \end{bem}

\begin{bem}
  For the purpose of this note, it would fully suffice to work with a weaker version of Theorem \ref{measure theorem}, a version asserting that the limit exists up to extracting subsequences---see \cite[Proposition 4.1]{ES} and \cite[Chapter 6]{book} for a proof of this weaker version.
\end{bem}

The Thurston measure $\FM_{\Thu}$ has many similarities with Lebesgue measure on $\BR^{6g-6}$. For example, it is homogenous of degree $6g-6$, meaning that for all $U\subset\CC(S)$ measurable we have
$$\FM_{\Thu}(t\cdot U)=t^{6g-6}\cdot\FM_{\Thu}(U)$$
for all $t\ge 0$. It follows that if $X\subset\CC(S)$ is a compact set such that for every $\alpha\in\CC(S)\setminus\{0\}$ there is a unique $t\in\BR_{>0}$ with $t\cdot \alpha\in X$, then there is a unique measure $\FN$ on $X$ such that the map
$$(\BR_{>0}\times X, t^{6g-7}\, dt\otimes\FN)\to(\CC(S)\setminus\{0\},\FM_{\Thu})$$ is a measure preserving homeomorphism. The measure $\FN(U)$ of $U\subset X$ is given by
$$\FN(U)=\lim_{t\searrow 1}\frac{\FM_{\Thu}([1,t]\times U)}{t-1}.$$
Combining all of this together we have:

\begin{lem}\label{lem:polar}
Let $X\subset\CC(S)$ be a compact set so that for every $\alpha\in\CC(S)\setminus\{0\}$ there is a unique $t\in\BR_{>0}$ with $t\cdot\alpha\in X$. There is a measure $\FN$ on the set $X$ such that the map
  $$(\BR_{>0}\times X, t^{6g-7}\, dt\otimes\FN)\to(\CC(S)\setminus\{0\},\FM_{\Thu})$$
is a measure preserving homeomorphism.\qed
\end{lem}

Continuous homogenous positive functions $F:\CC(S)\to\BR_{\ge 0}$ on the space of currents are a particular source of such sets $X$, where {\em positive} means that $F(\alpha)>0$ for all $\alpha\neq 0$ and {\em homogenous} means that $F(t\cdot\alpha)=t\cdot F(\alpha)$. Indeed, if we are given any such function $F$ then the set $X_F=F^{-1}(1)$ satisfies the condition in Lemma \ref{lem:polar}. Let then $\FN_F$ be the measure on $X_F$ provided by the said lemma. This  measure comes particularly in handy when integrating functions on $\CC(S)$ whose value at $\widehat\alpha$ solely depends on $F(\widehat\alpha)$. For example, with notation as in Lemma \ref{lem:polar}, we have
\begin{equation}\label{eq:polar1}
  \int_{\{F(\cdot)\le 1\}}F(\widehat\lambda)^n\ d\FM_{\Thu}(\widehat\lambda)=\frac 1{6g-6+n}\FN_F(X_F)
\end{equation}
for every $n\ge 0$. Taking $n=0$ we get in particular that
\begin{equation}\label{eq:polar2}
  \FN_F(X_F)=(6g-6)\cdot\FM_{\Thu}(\{F(\cdot)\le 1\}).
\end{equation}

The reason why we care about continuous, homogenous and positive functions on $\CC(S)$ is that there is such a function
$$\ell_S:\CC(S)\to\BR_{\ge 0}$$
satisfying $\ell_S(\widehat\gamma)=\ell_S(\gamma)$ for every closed geodesic $\gamma$, where the first $\ell_S(\cdot)$ is the function we are talking about and where the second one is just the length of the geodesic $\gamma$. The reader can see \cite{EPS} and \cite{DT} for other examples of continuous homogenous functions on the space of currents---in fact we will encounter yet other such functions below.

\section{}\label{sec5}

We are now ready to prove Theorem \ref{main}. The first step is to compute the total measure of the measures $m_L^{S,\gamma_0}$ introduced above:

\begin{lem}\label{lem:total measure}
  Let $S$ be a closed orientable surface of genus $g$ endowed with a negatively curved metric and denote by $\ell_S(\cdot)$ the associated length function. For every closed geodesic $\gamma_0$ we have
$$\lim_{L\to\infty}\frac 1{L^{6g-5}}\Vert\overline m_L^{S,\gamma_0}\Vert=\frac{6g-6}{6g-5}\cdot C(\gamma_0)\cdot\FM_{\Thu}(\{\ell_S(\cdot)\le 1\})$$
where $\overline m_L^{S,\gamma_0}$ is as in \eqref{eq:bowen2new} and where $C(\gamma_0)$ is as in Theorem \ref{measure theorem}.
\end{lem}

Note that the exponents of $L$ in Theorem \ref{measure theorem} and Lemma \ref{lem:total measure} differ, $6g-6$ versus $6g-5$.

\begin{proof}
Still writing $\gamma\sim\gamma_0$ to indicate that two closed unoriented geodesics are of the same type we have:
\begin{align*}
\frac 1{L^{6g-5}}\Vert\overline m_L^{S,\gamma_0}\Vert
&=\frac 1{L^{6g-5}}\sum_{\gamma\in\CS_L(S,\gamma_0)}\Vert\overline\gamma\Vert=\frac 1{L^{6g-5}}\sum_{\tiny\begin{array}{c}\gamma\sim\gamma_0\\ \ell_S(\gamma)\le L\end{array}}\ell_S(\gamma)\\
&=\frac 1{L^{6g-5}}\int_{\{\ell_S(\cdot)\le L\}}\ell_S(\cdot)\, d\left(\sum_{\gamma\sim\gamma_0}\delta_{\widehat\gamma}\right)\\
&=\int_{\{\ell_S(\cdot)\le L\}}\frac{\ell_S(\cdot)}L\, d\left(\frac 1{L^{6g-6}} \sum_{\gamma\sim\gamma_0}\delta_{\widehat\gamma}\right)\\
&=\int_{\{\ell_S(\cdot)\le 1\}}\ell_S(\cdot)\, d\left(\frac 1{L^{6g-6}}\sum_{\gamma\sim\gamma_0}\delta_{\frac 1L\widehat\gamma}\right)
\end{align*}
Now, Theorem \ref{measure theorem} asserts that the measures on the last line converge:
$$\lim_{L\to\infty}\frac 1{L^{6g-6}}\sum_{\gamma\sim\gamma_0}\delta_{\frac 1L\widehat\gamma}=C(\gamma_0)\cdot\FM_{\Thu}.$$
It follows that
\begin{align*}
  \lim_{L\to\infty}\frac 1{L^{6g-5}}\Vert\overline m_L^{S,\gamma_0}\Vert
  &=C(\gamma_0)\cdot\int_{\{\ell_S(\cdot)\le 1\}}\ell_S(\cdot)\, d\FM_{\Thu}\\
  &=\frac{6g-6}{6g-5}\cdot C(\gamma_0)\cdot\FM_{\Thu}(\{\ell_S(\cdot)\le 1\})
\end{align*}
where we have used \eqref{eq:polar1} and \eqref{eq:polar2} to get the last equality.
\end{proof}

Recall at this point the identification \eqref{eq-homeo} between the spaces $\CM_{\text{flip-flow}}(S)$ of flip and flow invariant measures and $\CC(S)$ of currents, and suppose that we are given a continuous function $f\in C^0(T^1S)$. Then we get a continuous function
$$\widehat f:\CC(S)\to\BR,\ \ \widehat f(\widehat\alpha)=\int_{T^1S}f(v)\, d\overline\alpha(v)$$
on the space of currents. It is continuous and homogenous. 

\begin{lem}\label{the lemma}
We have
\begin{multline*}
  \lim_{L\to\infty}\int_{T^1S}f\, d\left(\frac 1{\Vert\overline m_L^{S,\gamma_0}\Vert}\overline m_L^{S,\gamma_0}\right)=\\
  =\frac{6g-5}{(6g-6)\cdot\FM_{\Thu}(\{\ell_S(\cdot)\le 1\})}
\cdot\int_{\{\ell_S(\cdot)\le 1\}}\widehat f\ d\FM_{\Thu}\end{multline*}
for every $f\in C^0(T^1S)$.
\end{lem}

\begin{proof}
We compute as in Lemma \ref{lem:total measure}:
\begin{align*}
  \int_{T^1S}f\, d\left(\frac 1{\Vert\overline m_L^{S,\gamma_0}\Vert}\overline m_L^{S,\gamma_0}\right)
  &=\frac 1{\Vert\overline m_L^{S,\gamma_0}\Vert}\sum_{\tiny\begin{array}{c} \gamma\sim\gamma_0\\ \ell_S(\gamma)\le L\end{array}}\int_{T^1S}f\, d\overline\gamma\\
  &=\frac 1{\Vert\overline m_L^{S,\gamma_0}\Vert}\sum_{\tiny\begin{array}{c} \gamma\sim\gamma_0\\ \ell_S(\gamma)\le L\end{array}}\widehat f(\widehat\gamma)\\
  &=\frac L{\Vert\overline m_L^{S,\gamma_0}\Vert}\sum_{\tiny\begin{array}{c} \gamma\sim\gamma_0\\ \ell_S(\gamma)\le L\end{array}}\widehat f\left(\frac 1L\widehat\gamma\right)\\
  &=\frac L{\Vert\overline m_L^{S,\gamma_0}\Vert}\int_{\{\ell_S(\cdot)\le 1\}}\widehat f\ d\left(\sum_{\gamma\sim\gamma_0}\delta_{\frac1L\widehat\gamma}\right)
\end{align*}
In light of Lemma \ref{lem:total measure} we get that
\begin{equation}\label{eq:cat in the window}
  \begin{split}
    \int_{T^1S}f\, d\left(\frac 1{\Vert\overline m_L^{S,\gamma_0}\Vert}\overline m_L^{S,\gamma_0}\right)
    &\approx \frac K{C(\gamma_0)\cdot L^{6g-6}}\int_{\{\ell_S(\cdot)\le 1\}}\widehat f\ d\left(\sum_{\gamma\sim\gamma_0}\delta_{\frac1L\widehat\gamma}\right)\\
    &= \frac K{C(\gamma_0)}\int_{\{\ell_S(\cdot)\le 1\}}\widehat f\ d\left(\frac 1{L^{6g-6}}\sum_{\gamma\sim\gamma_0}\delta_{\frac1L\widehat\gamma}\right)
\end{split}
\end{equation}
where $\approx$ means that the ratio between both sides tends to $1$ when $L$ tends to $\infty$, and where
$$K=\frac{6g-5}{(6g-6)\cdot\FM_{\Thu}(\{\ell_S(\cdot)\le 1\})}.$$
Now, invoking again Theorem \ref{measure theorem} we get that the measure in the very end of \eqref{eq:cat in the window} converges to $C(\gamma_0)\cdot\FM_{\Thu}$, meaning that
$$\lim_{L\to\infty}\int_{\{\ell_S(\cdot)\le 1\}}\widehat f\ d\left(\frac 1{L^{6g-6}}\sum_{\gamma\sim\gamma_0}\delta_{\frac 1L\widehat\gamma}\right)=C(\gamma_0)\int_{\{l_S(\cdot)\le 1\}}\widehat f\ d\FM_{\Thu}.$$
  Putting all of this together we get that
  $$\lim_{L\to\infty}\int_{T^1S}f\, d\left(\frac 1{\Vert\overline  m_L^{S,\gamma_0}\Vert}\overline m_L^{S,\gamma_0}\right)= K\cdot\int_{\{\ell_S(\cdot)\le 1\}}\widehat f\ d\FM_{\Thu}$$
 as we had claimed.
\end{proof}

We are now basically done:

\begin{proof}[Proof of Theorem \ref{main}]
  It follows for example from the Riesz representation theorem (or, depending of your party affiliation, from the very definition of measure) that there is a unique measure $\FM$ on $T^1S$ with
  $$\int_{T^1S}f\, d\FM=\int_{\{\ell_S(\cdot)\le 1\}}\left(\int_{T^1S} f\, d\overline\lambda\right)d\FM_{\Thu}(\lambda)$$
  where, as usual, we denote by $\overline\lambda$ the flip and flow invariant measure associated to the measured lamination $\lambda\in\CM\CL(S)$.

  In those terms, Lemma \ref{the lemma} says that
  $$\lim_{L\to\infty}\int_{T^1S}f\, d\left(\frac 1{\Vert\overline m_L^{S,\gamma_0}\Vert}\overline m_L^{S,\gamma_0}\right)=\frac{6g-5}{(6g-6)\cdot\FM_{\Thu}(\{\ell_S(\cdot)\le 1\})}\cdot\int_{T^1S}f\ d\FM.$$
  This just means that the measures $\frac 1{\Vert\overline m_L^{S,\gamma_0}\Vert}\overline m_L^{S,\gamma_0}$ converge, when $L$ tends to $\infty$, to the measure
$$\FM^S=\frac{6g-5}{(6g-6)\cdot\FM_{\Thu}(\{\ell_S(\cdot)\le 1\})}\cdot\FM.$$
  Since, by \eqref{eq ricky gervais}, we have 
$$\frac 1{\Vert m_L^{S,\gamma_0}\Vert}m_L^{S,\gamma_0}=\frac 1{\Vert\overline m_L^{S,\gamma_0}\Vert}\overline m_L^{S,\gamma_0}$$
we are done.
\end{proof}

\section{}\label{sec6}

We are now ready to prove Theorem \ref{properties of measure}. We recall the statement for the convenience of the reader:

\begin{named}{Theorem \ref{properties of measure}}
The measure $\FM^S$ is not ergodic, has vanishing entropy, and its support has Hausdorff dimension $1$.  
\end{named}

\begin{proof}
  We start by proving that the measure $\FM^S$ is not ergodic. To begin with recall that the set
  $$\CU\CE=\{\lambda\in\CM\CL(S)\text{ with }\lambda\text{ uniquely ergodic}\}$$
  of uniquely ergodic measured lamniations has full Thurston measure by a result of Masur \cite{Masur}. Decompose the associated projective space
  $$P\CU\CE=\CU\CE/\BR_{>0}$$
  into two disjoint measurable sets
$$P\CU\CE=A_1\sqcup A_2,$$
  such that the associated subsets of $\CM\CL(S)$ have positive Thurston measure
\begin{equation}\label{eq-blablablabla}
\FM_{\Thu}(\{\lambda\in\CM\CL(S),\ [\lambda]\in A_i\})>0\text{ for }i=1,2.
\end{equation}
The sets $\CA_1,\CA_2\subset T^1S$ given by
$$\CA_i=\{v\in T^1S\text{ tangent to a leaf of the support of some }\lambda\in A_i\}$$
are then disjoint and invariant under the geodesic flow. In turn, \eqref{eq-blablablabla} implies that both have positive $\FM^S$-measure:
$$\FM^S(\CA_1),\FM^S(\CA_2)>0.$$
This proves that $\FM^S$ is not ergodic.
\medskip

Note that the support of $\FM^S$ is the set $\CK$ of vectors tangent to a leaf in the support of some measured lamination. That this set has Hausdorff dimension $1$ was proved by Fathi \cite[Theorem 3.1]{Fathi}, who also obtained in Corollary 3.4 of the same paper that the topological entropy of the restriction of the geodesic flow to $\CK$ vanishes---this implies that every measure supported by $\CK$, and in particular $\FM^S$, has vanishing entropy.
\end{proof}

\begin{bem}
  The fact the union of the supports of all measured laminations in $S$ has Hausdorff dimension $1$ is due to Birman and Series \cite{BS}. In \cite{Fathi} Fathi presents his Theorem 3.1, or more precisely its proof, as just giving a different perspective on the work of Birman and Series. The reason why we refer to Fathi's paper is that his more dynamical approach fits better what we do here.
\end{bem}

\section{}\label{sec7}
We will now prove Theorem \ref{theorem 3}. First recall that as in \eqref{eq-currents homeo} every homeomorphism
$$\phi:S\to S'$$
between negatively curved surfaces induces a homeomorphism 
$$\phi_*:\CC(S)\to\CC(S')$$
between the spaces of currents on the domain and target. Since we also have the homeomorphism \eqref{eq-homeo} between the space of currents and the space of flip and flow invariant measures on the unit tangent bundle, we get an induced homeomorphism
\begin{equation}\label{eq-homeo2}
  \phi_*:\CM_{\text{flip-flow}}(S)\to\CM_{\text{flip-flow}}(S').
\end{equation}
Our first goal is to figure out the image of $\FM^S$ under this map.

\begin{lem}\label{the lemma 2}
For every $f\in C^0(T^1S')$ we have
  $$\int_{T^1S'}f\, d\phi_*(\FM^{S})=\frac{6g-5}{6g-6}\cdot\FM_{\Thu}(\{\ell_S(\cdot)\le 1\})\cdot\int_{\{\ell_S(\phi^{-1}(\cdot))\le 1\}}\widehat f\, d\FM_{\Thu},$$
where the second integral is taken over the set
\begin{equation}\label{eq set}
  \{\ell_S(\phi^{-1}(\cdot))\le 1\}=\{\lambda\in\CM\CL(S')\vert\ell_S(\phi^{-1}(\lambda))\le 1\}
\end{equation}
of those measured laminations in $\CM\CL(S')$ which have length at most $1$ in the surface $S$.
\end{lem}

\begin{proof}
We start by tracking what the map $\phi_*$ does to the flip and flow invariant measure $\overline\gamma$ on $T^1S$ associated to a free homotopy class $\gamma$:
  $$\xymatrix{
    \CM_{\text{flip-flow}}(S) \ar[r] & \CC(S) \ar[r] & \CC(S') \ar[r] &\CM_{\text{flip-flow}}(S')\\
    \overline\gamma\ar[r] & \widehat\gamma \ar[r] & \widehat{\phi(\gamma)} \ar[r] & \overline{\phi(\gamma)}}$$
  Now, linearity of \eqref{eq-homeo2} implies that
  $$\phi_*(m^{S,\gamma_0}_L)=\sum_{{\tiny\begin{array}{c} \gamma\sim\gamma_0\\ \ell_S(\gamma)\le L\end{array}}}\overline{\phi(\gamma)}$$
  where $\gamma_0$ is some arbitrary but otherwise fixed closed geodesic in $S$.

  Then, exactly the same computation as we used in the first part of the proof of Lemma \ref{the lemma} shows that for any continuous function $f\in C^0(T^1S')$ we have
$$\int_{T^1S'}f\, d\phi_*\left(\frac 1{\Vert\overline m_L^{S,\gamma_0}\Vert}\overline m_L^{S,\gamma_0}\right)=\frac L{\Vert\overline m_L^{S,\gamma_0}\Vert}\int_{\{\ell_S(\phi^{-1}(\cdot))\le 1\}}\widehat f\, d\left(\sum_{\gamma\sim\phi(\gamma_0)}\delta_{\frac 1L\widehat\gamma}\right)$$
where the integral is taken over the set \eqref{eq set}. Still as in the proof of Lemma \ref{the lemma} we get that
$$\int_{T^1S'}f\, d\phi_*\left(\frac 1{\Vert\overline m_L^{S,\gamma_0}\Vert}\overline m_L^{S,\gamma_0}\right)\approx\frac{K(S)}{C(\gamma_0)}\int_{\{\ell_S(\phi^{-1}(\cdot))\le 1\}}\widehat f\, d\left(\frac 1{L^{6g-6}}\sum_{\gamma\sim\phi(\gamma_0)}\delta_{\frac 1L\widehat\gamma}\right)$$
where
$$K(S)=\frac{6g-5}{6g-6}\cdot\FM_{\Thu}(\{\ell_S(\cdot)\le 1\}).$$
Once at this point we get still as in Lemma \ref{the lemma} that
$$\lim_{L\to\infty}\int_{T^1S'}f\, d\phi_*\left(\frac 1{\Vert\overline m_L^{S,\gamma_0}\Vert}\overline m_L^{S,\gamma_0}\right)=K(S)\cdot\int_{\{\ell_S(\phi^{-1}(\cdot))\le 1\}}\widehat f\, d\FM_{\Thu}.$$
Since $\FM^{S}$ arises by Theorem \ref{main} as the limit of the measures $\frac 1{\Vert\overline m_L^{S,\gamma_0}\Vert}\overline m_L^{S,\gamma_0}$ and since $\phi_*$ is continuous, we get
$$\int_{T^1S'}f\, d\phi_*(\FM^{S})=K(S)\cdot\int_{\{\ell_S(\phi^{-1}(\cdot))\le 1\}}\widehat f\, d\FM_{\Thu},$$
as we had claimed.
\end{proof}

We are now going to rewrite the integral on the right in Lemma \ref{the lemma 2} using polar coordinates in $\CM\CL(S')$. To that end we choose a compact set $X\subset\CC(S')$ such that for every $\alpha\in\CC(S')\setminus\{0\}$ there is a uniquely determined $t>0$ with $t\cdot\alpha\in X$. Let $\FN$ be the measure on $X$ such that the map
$$(\BR_{>0}\times X,t^{6g-7}dt\otimes\FN)\to(\CC(S')\setminus\{0\},\FM_{\Thu})$$
is a measure preserving homeomorphism. Note that this induces a homeomorphism between
$$\{(t,\lambda)\in\BR_{>0}\times X\text{ with }0<t\le \ell_S(\phi^{-1}(\lambda))^{-1}\}\to\{\ell_S(\cdot)\le 1\}$$
We have then for every $f\in C^0(T^1S')$ that
\begin{align*}
\int_{\{\ell_S(\phi^{-1}(\cdot))\}}\widehat f\, d\FM_{\Thu}
&=\int_{\{\ell_S(\phi^{-1}(\cdot))\}}\left(\int_{T^1S'}f\, d\overline\lambda\right) d\FM_{\Thu}(\lambda)\\
&=\int_X\int_0^{\frac 1{\ell_S(\phi^{-1}(\lambda))}}\left(\int_{T^1S'}f\, d(\overline{t\cdot\lambda})\right) t^{6g-7}\cdot dt\cdot d\FN(\lambda)\\
&=\int_X\int_0^{\frac 1{\ell_S(\phi^{-1}(\lambda))}}\left(\int_{T^1S'}f\, d\overline\lambda\right) t^{6g-6}\cdot dt\cdot d\FN(\lambda)\\
&=\int_X \frac 1{(6g-5)\cdot\ell_S(\phi^{-1}(\lambda))^{6g-5}}\left(\int_{T^1S'}f\, d\overline\lambda\right) d\FN(\lambda)
\end{align*}
In light of Lemma \ref{the lemma 2} we get:
$$\int_{T^1S'}f\, d\phi_*(\FM^S)=\frac{\FM_{\Thu}(\{\ell_S(\cdot)\le 1\})}{6g-6}\int_X \frac 1{\ell_S(\phi^{-1}(\lambda))^{6g-5}}\left(\int_{T^1S'}f\, d\overline\lambda\right) d\FN(\lambda)$$
Applying this computation to $S=S'$ and $\phi=\Id$ we also get
$$\int_{T^1S'}f\, d\FM^{S'}=\frac{\FM_{\Thu}(\{\ell_{S'}(\cdot)\le 1\})}{6g-6}\int_X \frac 1{\ell_{S'}(\lambda)^{6g-5}}\left(\int_{T^1S'}f\, d\overline\lambda\right) d\FN(\lambda)$$
Let us now denote by $\CU\CE\subset\CM\CL(S')$ be the set of uniquely ergodic laminations and consider the function
$$\xi^{\phi,S}_{S'}:T^1S'\to\BR$$
given by
$$\xi^{\phi,S}_{S'}(v)=\left\{\begin{array}{ll}
\left(\frac{\ell_{S'}(\lambda)}{\ell_S(\phi^{-1}(\lambda))}\right)^{6g-5} & \text{ if }v\text{ is tangent to }\lambda\in\CU\CE\\
0 & \text{ otherwise}
\end{array}\right.$$
Since the set of uniquely ergodic measured laminations has full Thurston measure $\FM_{\Thu}$ \cite{Masur} we get then that
$$\int_X \frac 1{\ell_S(\phi^{-1}(\lambda))^{6g-5}}\left(\int_{T^1S'}f\, d\overline\lambda\right) d\FN(\lambda)=\int_X \frac 1{\ell_{S'}(\lambda)^{6g-5}}\left(\int_{T^1S'}f\xi^{\phi,S}_{S'}\, d\overline\lambda\right) d\FN(\lambda).$$
Taking all of this together we get
$$\frac 1{\FM_{\Thu}(\{\ell_S(\cdot)\le 1\})}\cdot\int_{T^1S'}f\, d\phi_*(\FM^S)=\frac 1{\FM_{\Thu}(\{\ell_{S'}(\cdot)\le 1\})}\cdot\int_{T^1S'}f\xi^{\phi,S}_{S'}\, d\FM^{S'}.$$
Since this holds true for all continuous functions $f\in C^0(T^1S')$ we get that indeed $\phi_*(\FM^S)$ is absolutely continuous with respect to $\FM^{S'}$ with Radon-Nikodym derivative
\begin{equation}\label{eq rnd}
  \frac{d\phi_*(\FM^S)}{d\FM^{S'}}=\frac{\FM_{\Thu}(\{\ell_{S}(\cdot)\le 1\})}{\FM_{\Thu}(\{\ell_{S'}(\cdot)\le 1\})}\cdot\xi^{\phi,S}_{S'}.
\end{equation}
Since the Radon-Nikodym derivative is essentially positive on the support of $\FM^{S'}$ we get that indeed both measures are in the same measure class. We have proved Theorem \ref{theorem 3}:

\begin{named}{Theorem \ref{theorem 3}}
  If $\phi:S\to S'$ is a homeomorphism between closed negatively curved surfaces then the measures $\phi_*(\FM^S)$ and $\FM^{S'}$ are in the same measure class.\qed
\end{named}

All that is left to do is to prove Theorem \ref{theorem 4}, which we do in the next section.

\section{}\label{sec8}

Let us remind the reader what we need to prove:

\begin{named}{Theorem \ref{theorem 4}}
  A homeomorphism $\phi:S\to S'$ between closed orientable hyperbolic surfaces is isotopic to an isometry if and only if $\phi_*(\FM^S)=\FM^{S'}$.
\end{named}
\begin{proof}
One direction is clear: if $\phi$ is an isometry, then $\phi_*(\FM^S)=\FM^{S'}$. Suppose conversely that $S$ and $S'$ are hyperbolic surfaces and that $\phi_*(\FM^S)=\FM^{S'}$. This means that the Radon-Nikodym derivative \eqref{eq rnd} is identically $1$. This means that we have
\begin{equation}\label{eq bla1}
  \ell_S(\lambda)=\frac{\FM_{\Thu}(\{\ell_{S}(\cdot)\le 1\})}{\FM_{\Thu}(\{\ell_{S'}(\cdot)\le 1\})}\cdot\ell_{S'}(\phi(\lambda))
\end{equation}
for every uniquely ergodic $\lambda\in\CM\CL(S)$. Density of the set of uniquely ergodic laminations implies that \eqref{eq bla1} holds for all $\lambda\in\CM\CL(S)\setminus\{0\}$. The following claim implies that this cannot happen unless $\phi$ is homotopic, and hence isotopic, to an isometry. 

\begin{claim}
  If $\phi:S\to S'$ is not homotopic to an isometry then there are $\alpha,\beta\in\CM\CL(S)$ with $\ell_S(\alpha)<\ell_{S'}(\phi(\alpha))$ and $\ell_S(\beta)>\ell_{S'}(\phi(\beta))$.
\end{claim}
\begin{proof}
  Let $\Lip(\phi)$ denote the infimum of the Lipschitz constants over all maps homotopic to $\phi$. It is due to Thurston \cite{Thurston} that $\Lip(\phi)>1$ unless $\phi$ is homotopic to an isometry and that
  $$\Lip(\phi)=\max_{\lambda\in\CM\CL(S)\setminus\{0\}}\frac{\ell_{S'}(\phi(\lambda))}{\ell_S(\lambda)}.$$
The existence of $\alpha$ follows. The existence of $\beta$ follows as well once we repeat the argument replacing $\phi$ by $\phi^{-1}$. 
\end{proof}

Having proved the claim, we have proved that if $S$ and $S'$ are hyperbolic and if $\phi(\FM^S)=\FM^{S'}$ then $\phi$ is homotopic to an isometry. 
\end{proof}

Let us now construct an example showing that Theorem \ref{theorem 4} fails if say $S'$ is allowed to have variable curvature, even after we normalize the area to be equal to that of $S$. Let us start by choosing $S$ hyperbolic and recall that, by Theorem \ref{properties of measure}, the support $\CK\subset T^1S$ of $\FM^S$ has Hausdorff dimension $1$. Its projection $\CK'\subset S$ to the surface is closed and also has Hausdorff dimension $1$. Denoting by $\sigma_0$ the hyperbolic metric on $S$ let $\sigma$ be another metric (with same area) obtained by very slightly perturbating $\sigma_0$ on some open set whose closure does not meet $\CK'$. Let $S'=(S,\sigma)$ and $\phi=\Id$. The map
$$\phi:(S,\sigma_0)\to(S,\sigma)$$
is then an isometry on the set $\CK'$ and hence satisfies that $\ell_{S'}(\phi(\lambda))=\ell_S(\lambda)$ for all $\lambda\in\CM\CL(S)$. It follows that the Radon-Nikodym derivative \eqref{eq rnd} of $\phi_*(\FM^S)$ with respect to $\FM^{S'}$ is identically one, meaning that $\phi_*(\FM^S)=\FM^{S'}$. On the other hand there is no isometry from a hyperbolic surface to one with non-constant curvature.

\section{}\label{sec9}

So far we have been considering closed surfaces. We explain briefly how to modify the proofs of the theorems in the introduction so that they apply to surfaces with cusps. Actually, we will just focus on Theorem \ref{main} leaving to the reader the pleasure of modifying the others.

Let us thus suppose that $S$ is a complete hyperbolic surface of finite area and let $\gamma_0$ be a closed primitive geodesic in $S$. We also suppose that $S$ is not a thrice punctured sphere---all results are trivial in that case.

The first basic fact (see for example \cite[Lemma 2.8]{book}) we need is that there is a compact subsurface $S_0\subset S$ which contains every closed geodesic $\gamma$ with at most as many self-intersections as $\gamma_0$:
$$\iota(\gamma,\gamma)\le\iota(\gamma_0,\gamma_0).$$
Since two curves of the same type have the same self-intersection number we get that all geodesics of the same type as $\gamma_0$ are contained in $S_0$. It follows that the measures $m_L^{S,\gamma_0}$ and $\overline m_L^{S,\gamma_0}$ are supported by $S_0$ for all $L$.

This means that in the proofs one does not need to worry about the measure wandering off to infinity. But this is only one of the issues one faces when the surface is non-compact. The point is that for technical reasons, when working with currents (and wanting to have a continuous and homogenous extension of $\iota(\cdot,\cdot)$ to the space of currents) one is forced by nature to work on compact surfaces. However, while the surfaces have to be compact, they can have boundary. 

Let then $S'\subset S$ be a compact subsurface such that
\begin{itemize}
\item $S_0\subset S'\setminus\D S'$, and
\item $S\setminus S'$ is a union of annuluar components, one for each cusps of $S$.
\end{itemize}
Endow $S'$ with a negatively curved metric with totally geodesic boundary and which agrees with the metric of $S$ near $S_0$.

By definition, a current on $S'$ is a $\pi_1(S')$-invariant Radon measure on the space $\CG(\tilde S')$ of complete geodesics on the universal cover of $S'$. Denote by $\CC_{S_0}(S')$ be the set of currents on $S'$ supported by geodesics which project into $S_0$. For example, the current $\widehat\gamma$ belongs to $\CC_{S_0}(S')$ for every $\gamma$ of the same type as $\gamma_0$.

Note also that $S'$-geodesics contained in $S_0$ are, by the choice of the metric on $S$, also $S$-geodesics. It follows that we can consider currents $\widehat\mu\in\CC_{S_0}(S')$ as $\pi_1(S)$-invariant measures on the set of geodesics on $S$. We thus get a continuous map
\begin{equation}\label{sviedyqd}
  \CC_{S_0}(S')\to\CM_{\text{flip-flow}}(S),\ \ \widehat\mu\mapsto\overline\mu=dt\otimes\widehat\mu
  \end{equation}
where the measure $\overline\mu$ is still given by \eqref{eq-product}. This maps is also proper.

Now, Theorem \ref{measure theorem} is proved in \cite{book} not only for closed but also for general compact connected orientable surfaces of negative Euler characteristic (other than the pair of pants). Applying it to $\gamma_0$ when seen as a curve in $S'$ we get that
\begin{equation}\label{eq measure theorem 2}
  \lim_{L\to\infty}\frac 1{L^{6g-6+2r}}\sum_{\gamma\sim\gamma_0}\delta_{\frac 1L\widehat\gamma}=C(\gamma_0)\cdot\FM_{\Thu}.
  \end{equation}
Here $g$ and $r$ are the genus and number of cusps of $S$ and the convergence takes place with respect to the weak-*-topology on the space of Radon measures on $\CC_{S_0}(S')$. 

Armed with \eqref{sviedyqd} and \eqref{eq measure theorem 2} we can repeat word by word the proof of Lemma \ref{lem:total measure} and Lemma \ref{the lemma}---Theorem \ref{main} follows.
\medskip

  The reader might be wondering if this is actually true. Earlier we were making a big deal of the fact that the space of currents is homeomorphic to that of flip and flow invariant measures, and the map \eqref{sviedyqd} is obviously not a homeomorphism. If the reader is worried about this, then they can really go back to the proof of Lemma \ref{lem:total measure} and Lemma \ref{the lemma} and check that we are only using the continuity and properness of \eqref{eq-homeo}.

\end{document}